\documentclass{article}

\RequirePackage[hyphens]{url}
\def\doi#1{\url{http://dx.doi.org/#1}}

\RequirePackage{amsmath}
\RequirePackage{amsthm}
\RequirePackage{amsfonts}
\RequirePackage{amssymb}

\newif\iftrimmarks  \trimmarksfalse \trimmarkstrue

\theoremstyle{theorem}

\newtheorem{proposition}{Proposition}
\newtheorem{lemma}{Lemma}

\theoremstyle{definition}

\newtheorem*{remark}{Remark}

\begin{document}

\title{An Interpolation Approach to $\zeta(2n)$.\\ A longer version for students}

\author{Samuel G. Moreno, Esther M. Garc\'{\i}a--Caballero\\{\small{\it Departamento de Matem\'aticas,
Universidad de Ja\'en}}\\{\small{\it 23071 Ja\'en, Spain}}}

\date{}

\maketitle

\begin{abstract}

We solve an interpolation problem for computing $\zeta(2n)$ in a rather elementary way, by generalizing the main idea in \cite{SE}.

\vspace{.5cm}

\noindent {\it 2010 Mathematics Subject Classification}: Primary 40A25; Secondary 11M99.

\noindent  {\it Keywords}: Riemann Zeta function; Lidstone interpolation; Summation of series

\end{abstract}

 \section{Starting point and main idea}\label{s01}

We begin by recalling a well-known trigonometric identity, also called the {\it Dirichlet kernel formula}, which reads \cite[p.294]{Da}
\begin{eqnarray}\label{e1005}
\frac{1}{2}+\sum_{k=1}^{m} \cos (k x)=\frac{\sin \big((m+1/2)x\big)}{2 \sin (x/2)},\qquad (m=1,2,\ldots).
\end{eqnarray}
\noindent Formula (\ref{e1005}) can be easily proved by first multiplying $\cos (kx)$  by $\sin (x/2)$ and then using that $\cos a \sin b= (\sin (a+b)-\sin (a-b))/2$; in closing, the magic touch: sum up from  $k=1$ to $k=m$ and \ldots  enjoy telescoping!

Now we come to the steps in which the computation of $\zeta(2n)$ will be developed:
\begin{itemize}
  \item Multiplying (\ref{e1005}) by an algebraic polynomial $q$, and then integrating over $[0,\pi]$, we obtain
  \begin{eqnarray*}\label{e1010}
 \sum_{k=1}^{m}\int_0^{\pi}q(x) \cos (k x)\,dx&=&-\frac{1}{2}\int_0^{\pi}q(x) \,dx+\int_0^{\pi}q(x)\frac{\sin \big((m+1/2)x\big)}{2 \sin (x/2)}\,dx.
  \end{eqnarray*}
  \item For each positive integer $n$, the polynomial $q$ is chosen to get $\int_0^{\pi}q(x) \cos (k x)\,dx=1/k^{2n}$, which yields
  \begin{eqnarray}\label{e1015}
 \sum_{k=1}^{m}\frac{1}{k^{2n}}&=&-\frac{1}{2}\int_0^{\pi}q(x) \,dx+\int_0^{\pi}q(x)\frac{\sin \big((m+1/2)x\big)}{2 \sin (x/2)}\,dx.
  \end{eqnarray}
  \item Finally, after doing the computations (at the right-hand side of (\ref{e1015})) and taking limits as $m\to \infty$, our hope is to find a closed form for $\zeta(2n)$ by means of
 \begin{eqnarray}\label{e1020}
 \zeta(2n)&=&\sum_{k=1}^{\infty}\frac{1}{k^{2n}}\nonumber\\&=&-\frac{1}{2}\int_0^{\pi}q(x) \,dx+\lim_{m\to \infty}\left(\int_0^{\pi}q(x)\frac{\sin \big((m+1/2)x\big)}{2 \sin (x/2)}\,dx\right).
  \end{eqnarray}
\end{itemize}

Well, our plan is music to our ears provided that we can find a polynomial $q$ such that $\int_0^{\pi}q(x) \cos (k x)\,dx=1/k^{2n}$, and provided also that we can obtain the above limit. These two issues constitute the target of the next two sections.

\section{First obstacle: Does any polynomial matching our plan exist?}\label{s02}

\subsection{A notoriously simple formula}

Let us first examine how  $\int_0^{\pi}q(x) \cos (k x)\,dx$ can be calculated and its value expressed in terms of the derivatives $q^{(2j-1)}(\pi)$ and $q^{(2j-1)}(0)$.
\begin{proposition}\label{p201}
If $k$ is a positive integer and $q$ is a polynomial, then
\begin{eqnarray}\label{e2005}
\int_0^{\pi} q(x) \cos kx \,dx=\sum_{j=1}^{\infty}(-1)^{j+1}\frac{(-1)^kq^{(2j-1)}(\pi)-q^{(2j-1)}(0)}{k^{2j}}.
\end{eqnarray}
\noindent (Notice that the above sum always terminates, since $q$ is a polynomial and thus $q^{(2j-1)}=0$ for all $j$ such that $2j-1>{\rm deg} (q)$.)
\end{proposition}
\begin{proof} Integrating by parts we have
\begin{eqnarray}\label{e2010}
\int_0^{\pi}q(x) \cos k x\,dx&=&q(x)\frac{\sin k x}{k}\Bigg|_0^{\pi}
-\frac{1}{k}\int_0^{\pi}q'(x)\sin k x \,dx \nonumber\\&=&-\frac{1}{k}\int_0^{\pi}q'(x)\sin k x \,dx\nonumber\\
&=&\frac{1}{k}\left(q'(x)\frac{\cos k x}{k}\Bigg|_0^{\pi}-\frac{1}{k}\int_0^{\pi}q''(x)\cos k x \,dx\right)\nonumber\\
&=&\frac{(-1)^k q'(\pi)-q'(0)}{k^2}-\frac{1}{k^2}\int_0^{\pi}q''(x)\cos k x \,dx.
\end{eqnarray}
\noindent Hence, by iteration, with $q$ replaced by $q''$ in the above formula, we get
\begin{eqnarray*}\label{e2015}
\int_0^{\pi}q(x)\cos k x\,dx&=&\frac{(-1)^k q'(\pi)-q'(0)}{k^2}\nonumber\\&-&\frac{1}{k^2}\left(\frac{(-1)^k q^{(3)}(\pi)-q^{(3)}(0)}{k^2}-\frac{1}{k^2}\int_0^{\pi}q^{(4)}(x)\cos k x \,dx\right).
\end{eqnarray*}
\noindent A new step in our deduction is obtained by following the same procedure, now repla\-cing $q$ by $q^{(4)}$ in (\ref{e2010}) to obtain
\begin{eqnarray*}\label{e2020}
\int_0^{\pi}q(x)\cos k x\,dx&=&\frac{(-1)^k q'(\pi)-q'(0)}{k^2}-\frac{(-1)^k q^{(3)}(\pi)-q^{(3)}(0)}{k^4}\\&+&\frac{1}{k^4}\left(\frac{(-1)^k q^{(5)}(\pi)-q^{(5)}(0)}{k^2}-\frac{1}{k^2}\int_0^{\pi}q^{(6)}(x)\cos k x \,dx\right)\\&=&\frac{(-1)^k q'(\pi)-q'(0)}{k^2}-\frac{(-1)^k q^{(3)}(\pi)-q^{(3)}(0)}{k^4}\\& &+\frac{(-1)^k q^{(5)}(\pi)-q^{(5)}(0)}{k^6}-\frac{1}{k^6}\int_0^{\pi}q^{(6)}(x)\cos k x \,dx.
\end{eqnarray*}
\noindent Following in the same manner (with a formal inductive argument if desired), the proof is complete.
\end{proof}

\subsection{The pleasant consequence}

In view of Proposition \ref{p201} we deduce that for each positive integer $n$, if a polynomial $P_{2n}$ of exact degree $2n$ exists such that  $P'_{2n}(0)=P'''_{2n}(0)=\cdots=P_{2n}^{(2n-3)}(0)=0$, also $P_{2n}^{(2n-1)}(0)=(-1)^n$, and $P'_{2n}(\pi)=P'''_{2n}(\pi)=\cdots=P_{2n}^{(2n-1)}(\pi)=0$, then (\ref{e2005}) simplifies to
\begin{eqnarray}\label{e2025}
\int_0^{\pi} P_{2n}(x) \cos k x \,dx=\frac{1}{k^{2n}},\qquad (k=1,2,\ldots),
\end{eqnarray}
\noindent and this is, exactly, what we were looking for.

\subsection{Intermezzo. A journey to Lidstoneland}

Summarizing, for a fixed positive integer $n$, we are thus led to find a polynomial $P_{2n}$ such that  $P^{(2j-1)}(0)=(-1)^n\delta_{j\,n}$ and $P^{(2j-1)}(\pi)=0$ for $1\leq j\leq n$. (We have used the {\it Kronecker delta} $\delta_{j\,n}$ which equals 0 if $j\neq n$ and which satisfies that $\delta_{n\,n}=1$.) This situation reminds us the celebrated problem of polynomial interpolation of finding the unique algebraic polynomial  $Q_n$,  of degree at most $n$, for which $Q_n(x_i)=y_i$, where the $x_i$s and the $y_i$s are $n+1$ given values with $x_i\neq x_j$ for $i\neq j$. But we can find many more interpolation problems that admit unique solution. For example, the so-called {\it Taylor interpolation}, which consists in finding the unique polynomial $Q_n$ with ${\rm deg} (Q_n)\leq n$ for which $Q_n^{(j)}(x_0)=y_j$, where $0\leq j\leq n$, and  where the point $x_0$ and the $y_j$s are given numbers. If conditions on the consecutive derivatives are given not just for a single point $x_0$ , but for two points $x_0$, $x_1$, we will be facing the so-called  {\it two point Taylor interpolation}. This is again a well posed interpolation problem with a single solution. We refer the reader to the very classic source (and still an excellent one) \cite[Ch. I]{Da}, in which the interpolation problem is fully and deeply analyzed, and where many more examples can be found.

A not so popular (but nevertheless very well studied) interpolation situation emerges when the conditions are given for two points but, instead of the consecutive derivatives, now the given derivatives are just the even ones. This problem was originally posed and solved by Lidstone in \cite{L}, and extended by many other authors (see, for example, \cite{W}). In the remains of this intermezzo we show the very basic facts on Lidstone interpolation that we will need in order to overcome our first obstacle. The interested reader can verify each one of the claims without any further reading.

The unique polynomial $p$ of degree at most $2n-1$ such that $p^{(2j)}(0)=a_{j}$ and $p^{(2j)}(1)=b_{j}$ for $0\leq j\leq n-1$ is
\begin{eqnarray}\label{2030}
p(x) &=&  \sum_{j=0}^{n-1}\left(a_{j}\Lambda_{j}(1-x)+b_{j}\Lambda_{j}(x)\right),
\end{eqnarray}
\noindent where the so-called {\it Lidstone polynomials} $\Lambda_{k}$ are recursively defined by
\begin{alignat*}{1}\label{2035}
\Lambda_{0}(x)&=x,\\
\Lambda_{k}''(x)&=\Lambda_{k-1}(x),\\
\Lambda_{k}(0)&=\Lambda_{k}(1)=0,\qquad (k=1,2,\ldots).
\end{alignat*}

Define $e_{k}(x)=x^k$. Since $e_{2k+1}^{(2j)}(0)=0$ and $e_{2k+1}^{(2j)}(1)=\displaystyle{\frac{(2k+1)!}{(2k+1-2j)!}}$  for  $0\leq j\leq k$, then
\begin{eqnarray*}\label{2040}
x^{2k+1} &=& \sum_{j=0}^{k} \frac{(2k+1)!}{(2k+1-2j)!}\Lambda_{j}(x), \qquad (k=0,1,\ldots),
\end{eqnarray*}
\noindent which yields the recursion
\begin{eqnarray}\label{2045}
\Lambda_{k}(x)=\frac{x^{2k+1}}{(2k+1)!}- \sum_{j=0}^{k-1} \frac{\Lambda_{j}(x)}{(2k+1-2j)!}.
\end{eqnarray}
\noindent Well, that is what we need. No hard machinery after all. Now it's time to meet the polynomials $P_{2n}$ that fulfill (\ref{e2025}).

\subsection{Passing the first challenge with flying colors} For each $n=1,2,\ldots$, our task is to find a polynomial $P_{2n}$ such that
\begin{eqnarray*}
P_{2n}^{(2j-1)}(0)=(-1)^n\delta_{j\,n}\qquad \mbox{ and }\qquad P_{2n}^{(2j-1)}(\pi)=0,\qquad \mbox{ for }\qquad j=1,\ldots,n.
\end{eqnarray*}
\noindent But such a polynomial exists. Define $Q_n(x)=P_{2n}'(\pi x)$ and note that conditions above transform to
\begin{eqnarray*}\label{e2050}
Q_n^{(2j)}(0) &=& \pi^{2j}P_{2n}^{(2j+1)}(0)=\pi^{2j}P_{2n}^{(2(j+1)-1)}(0)=(-1)^n\pi^{2j}\delta_{j+1\,n},\\
Q_n^{(2j)}(1)&=& \pi^{2j}P_{2n}^{(2j+1)}(\pi)=\pi^{2j}P_{2n}^{(2(j+1)-1)}(\pi)=0, \qquad  \qquad (1\leq j+1\leq n).
\end{eqnarray*}
\noindent Hence, by (\ref{2030}) we deduce
\begin{eqnarray*}
 Q_n(x)=(-1)^n\pi^{2n-2}\Lambda_{n-1}(1-x).
\end{eqnarray*}

Let us mention an important detail. Since $2n$ conditions are the ones for $P_{2n}$ to satisfy, the degree of $P_{2n}$ is one unit more than required. In other words, if we insist to look for a polynomial of precise degree $2n$, then it is still possible to add one more condition on $P_{2n}$.  Just for the sake of simplicity, let us impose $P_{2n}(0)=0$. This choice yields
\begin{eqnarray}\label{e2055}
P_{2n}(x)&=&\int_{0}^{x} Q_n\left(\frac{t}{\pi}\right) \,dt=(-1)^n\pi^{2n-2}\int_{0}^{x} \Lambda_{n-1}\left(1-\frac{t}{\pi}\right) \,dt\nonumber\\&=&(-1)^n\pi^{2n-1}\left(\Lambda_{n}'(1)-\Lambda_{n}'\left(1-\frac{x}{\pi}\right)\right),
\end{eqnarray}
\noindent and thus
\begin{eqnarray}\label{e2060}
\int_{0}^{\pi} P_{2n}(x)\,dx=(-1)^n\pi^{2n}\Lambda_{n}'(1).
\end{eqnarray}

Finally notice that since $\Lambda_{n}$ is a polynomial of precise degree $2n+1$ (see (\ref{2045})), then we deduce from (\ref{e2055}) that ${\rm deg} P_{2n}=2n$.

\section{Second obstacle: How to compute the limit of an integral without computing the integral}\label{s03}

\noindent As mentioned in Section \ref{s01}, the success of our plan strongly depends on our ability to handle the sum
\begin{eqnarray}\label{e3005}
-\frac{1}{2}\int_0^{\pi}P_{2n}(x) \,dx+\lim_{m\to \infty}\left(\int_0^{\pi}P_{2n}(x)\frac{\sin \big((m+1/2)x\big)}{2 \sin (x/2)}\,dx\right).
\end{eqnarray}
\noindent (This is the right-hand side of formula (\ref{e1020}), with $q$ replaced by $P_{2n}$.) Notice that we have just calculated the first term in the above sum (it  equals minus half the value in (\ref{e2060})), so now we have to consider only the second term. But we must emphasize that our concern is not to compute
\begin{eqnarray*}\label{e3010}
\int_{0}^{\pi} \frac{P_{2n}(x)}{2 \sin (x/2)}\sin ((m+1/2) x)\,dx,
\end{eqnarray*}
\noindent  but its limit as $m\to \infty$. To this end, we have at hand a result that fits to a T:
\begin{lemma}[Riemann-Lebesgue] (see \cite[p. 296]{Da})
If $f$ is a function with continuous derivative on $[0,\pi]$, then
\begin{eqnarray*}\label{e3015}
\lim_{t\to \infty}\left(\int_0^{\pi}f(x) \cos (t x)\,dx\right)=\lim_{t\to \infty}\left(\int_0^{\pi}f(x) \sin (t x)\,dx\right)=0.
\end{eqnarray*}
\end{lemma}

The proof of the above ``easy version'' of the Riemann-Lebesgue lemma is nothing but integrating by parts. Let us show how this famous results works to compute the limit in (\ref{e3005}). The trick consists in writing $P_{2n}(x)=x R_{2n-1}(x)$ ($R_{2n-1}$ being a polynomial of degree $2n-1$),  which is possible since by (\ref{e2055}) we have that $P_{2n}(0)=0$.  This factorization yields
\begin{eqnarray*}\label{e3020}
\frac{P_{2n}(x)}{2 \sin (x/2)}=\frac{x/2}{\sin(x/2)} R_{2n-1}(x),
\end{eqnarray*}
\noindent and the continuity of $P_{2n}(x)/ (2\sin (x/2))$ and its derivative on $[0,\pi]$ follows readily. Therefore,
\begin{eqnarray}\label{e3025}
\lim_{m\to \infty}\left(\int_0^{\pi}P_{2n}(x)\frac{\sin \big((m+1/2)x\big)}{2 \sin (x/2)}\,dx\right)&=&0.
\end{eqnarray}

\section{Epilogue}\label{s04}

The good news is that all the workings have already been done. Indeed, substituting (\ref{e2060}) and (\ref{e3025}) into (\ref{e1020}), and taking into account (\ref{2045}), we  conclude that

\begin{eqnarray}\label{e4005}
\zeta(2n)=\sum_{k=1}^{\infty} \frac{1}{k^{2n}}&=&(-1)^{n+1}\frac{\pi^{2n}\Lambda_{n}'(1)}{2},\qquad (n=1,2,\ldots),
\end{eqnarray}
\noindent where  $\Lambda_{0}'(1)=1$ and $\Lambda_{n}'(1)=\displaystyle{\frac{1}{(2n)!}- \sum_{j=0}^{n-1} \frac{\Lambda_{j}'(1)}{(2n+1-2j)!}}$. The first few values of $\Lambda_{n}'(1)$, $n=1,\ldots,5$ are $1/3,-1/45,2/945,-1/4725,2/93555$. One can go on computing them as long as patience will permit.

\begin{remark}{}
By changing the conditions on $P_{2n}$, the same problem of summing up $\sum_{k=1}^{\infty} (1/k^{2n})$ could have been solved in a different way. For example, the conditions
$P_{2n}(0)=0$, $P_{2n}^{(2j-1)}(0)= P_{2n}^{(2j-1)}(\pi)=0$ for $1\leq j\leq n-1$, $P_{2n}^{(2n-1)}(\pi)=0$ and $P_{2n}^{(2n)}(0)=(2n)!$ define uniquely the polynomials $P_{2n}$ by means of
\begin{eqnarray*}
 P_{2n}(x)=\sum_{k=0}^{2n-1}\binom{2n}{k} (2\pi)^k B_k\, x^{2n-k},
\end{eqnarray*}
\noindent which are intimately related with Bernoulli polynomials (the $B_k$ are Bernoulli numbers as described in \cite{AS}). In some sense, this resembles the approach in \cite{C} to compute $\zeta(2n)$.

In much a similar way, Euler polynomials could have appeared. In fact, this is the approach in \cite{D}.

In closing, let us mention that formula (4.6) in \cite{W} has some resemblance with our (\ref{e4005}) (we came across with \cite{W} after the submission of this note). However, it is worth mentioning that our method is much more elementary, and also that we give the key to recursively compute $\zeta(2n)$, by relating $\Lambda_{n}'(1)$ with  $\Lambda_{j}'(1)$ for $j=0,1,\ldots,n-1$.
\end{remark}

\vfill\eject

\end{document}